\documentclass[11pt]{amsart}
\usepackage{amsmath,amsthm}
\usepackage{amssymb}
\usepackage{pinlabel}
\usepackage[left = 3cm, right = 3 cm , top = 3 cm , bottom = 3cm, marginparwidth=2.5cm ]{geometry}
\usepackage[utf8]{inputenc}
\usepackage{cite}
\usepackage[textsize=small]{todonotes}
\usepackage{enumerate}
\usepackage{mathtools}
\usepackage{tikz}
\usepackage[hidelinks]{hyperref}

\newtheorem{thm}{Theorem}[section]
\newtheorem{prop}[thm]{Proposition}
\newtheorem{lem}[thm]{Lemma}

\newtheorem{notation}[thm]{Notation}
\newtheorem*{eulerclassone}{Euler Class One Conjecture}

\theoremstyle{remark}
\newtheorem{remark}[thm]{Remark}

\newtheorem{ex}[thm]{Example}

\theoremstyle{definition}
\newtheorem{definition}[thm]{Definition}

\newcommand\C{\mathbb{C}}

\newcommand\R{\mathbb{R}}

\newcommand\rot{\operatorname{rot}}

\tikzset{every picture/.style=thick}
\usetikzlibrary{decorations.markings,intersections}

\newcommand\blfootnote[1]{%
\begingroup
\renewcommand\thefootnote{}\footnote{#1}%
\addtocounter{footnote}{-1}%
\endgroup
}

\begin{document}

\title{Thurston norm and Euler classes of tight contact structures}

\author[Steven Sivek]{Steven Sivek}
\address{Department of Mathematics\\Imperial College London}
\email{s.sivek@imperial.ac.uk}
\address{Max Planck Institute for Mathematics}
\email{ssivek@mpim-bonn.mpg.de}

\author[Mehdi Yazdi]{Mehdi Yazdi}
\address{Department of Mathematics\\King's College London}
\email{mehdi.yazdi@kcl.ac.uk}

\date{}
\thanks{}

\begin{abstract}
Bill Thurston proved that taut foliations of hyperbolic 3-manifolds have Euler classes of norm at most one, and conjectured that any integral second cohomology class of norm equal to one is realised as the Euler class of some taut foliation. Recent work of the second author, joint with David Gabai, has produced counterexamples to this conjecture.  Since tight contact structures exist whenever taut foliations do and their Euler classes also have norm at most one, it is natural to ask whether the Euler class one conjecture might still be true for tight contact structures.
In this short note, we show that the previously constructed counterexamples for Euler classes of taut foliations in \cite{yazdi2020thurston} are in fact realised as Euler classes of tight contact structures. This provides some evidence for the Euler class one conjecture for tight contact structures.
\end{abstract}

\maketitle

\section{Introduction} \blfootnote{MSC 2020 Subject Classification: 57K31, 57K33, 57K32}%
For any compact, orientable 3-manifold $M$, Thurston defined a semi-norm on the second real homology groups $H_2(M)$ and $H_2(M , \partial M)$. For a  compact surface $S$ with components $S_1 , \cdots, S_n$, define the \emph{negative part of the Euler characteristic} as 
\begin{eqnarray*}
\chi_-(S) = \sum_{i=1}^{n} \max\{ 0 , - \chi(S_i) \}.
\end{eqnarray*}

In words, $\chi_-(S)$ is the absolute value of the Euler characteristic after discarding any disc and sphere components. The \emph{Thurston norm} of an integral class $a \in H_2(M)$ (or $a \in H_2(M, \partial M)$) is defined by taking the minimum of $\chi_-(S)$ over all (properly) embedded compact orientable surfaces $S \subset M$ with homology class $[S]=a$. The norm of a rational point $a$ is defined by scaling, and the norm of a real point $a$ is defined continuously. The Thurston norm on $H_2(M)$ and $H_2(M , \partial M)$ defines norms on the respective dual vector spaces $H^2(M)$ and $H^2(M , \partial M)$. 

By a \emph{foliation} of a 3-manifold, we mean a two-dimensional foliation, i.e., a decomposition of the 3-manifold into injectively immersed surfaces that locally has the structure of a product $\mathbb{R}^2 \times \mathbb{R}$ with leaves $\mathbb{R}^2 \times \text{point}$. A foliation is \emph{transversely orientable} if there is a consistent choice of transverse orientations on the leaves. By Wood, on a closed orientable 3-manifold, every transversely orientable plane field is homotopic to the tangent plane field of a foliation \cite{wood1969foliations}. In particular, Wood's theorem implies that when $M$ is closed and orientable, every cohomology class $a \in 2 H^2(M ; \mathbb{Z})$ is realised as the Euler class of the tangent plane bundle to a transversely oriented foliation on $M$. See \cite[Theorem 8.1]{yazdi2020thurston} for this deduction, originally due to Wood.

A transversely oriented foliation of a 3-manifold is \emph{taut} if for every leaf $L$ there is a circle $c_L$ intersecting $L$ and transverse to the foliation. Manifolds that admit taut foliations have similar properties to hyperbolic 3-manifolds; for example if $M \neq S^2 \times S^1$ is orientable then $M$ is irreducible \cite{novikov1965topology,rosenberg1968foliations} and has infinite fundamental group \cite{novikov1965topology}. 

From now on we assume that $M$ is closed, orientable, and irreducible; there are generalisations for manifolds with boundary.  Thurston \cite{thurston1986norm} studied the relation between taut foliations and Thurston norm and proved his celebrated theorem that compact leaves of taut foliations are \emph{norm-minimising}. He showed this by establishing the inequality
\begin{eqnarray}
|\langle e(\mathcal{F}) , [S]\rangle | \leq |\chi_-(S)|, 
\label{inequality}
\end{eqnarray}
where $S$ is any compact embedded oriented surface and $e(\mathcal{F}) \in H^2(M ; \mathbb{R})$ is the Euler class of the oriented tangent plane bundle of the taut foliation $\mathcal{F}$. In other words, the Euler class of a taut foliation has dual norm at most one. Note this is in contrast with the case of general foliations as discussed previously. An integral class $a \in H^2(M ; \mathbb{Z})$ satisfies the \emph{parity condition} if $a \in 2 H^2(M ; \mathbb{Z})$. In particular, the Euler class of a transversely oriented plane field over a compact orientable 3-manifold satisfies the parity condition. See \cite[Proposition 3.12]{yazdi2020thurston} for this deduction. Thurston conjectured the following as a converse to his theorem, where $x^*$ denotes the dual norm. Thurston was aware of the parity condition, and so we implicitly assume the parity condition as part of the hypotheses of the conjecture. 

\begin{eulerclassone}[Thurston, 1976]
Let $M$ be a closed, orientable, irreducible and atoroidal 3-manifold and let $a$ be any integral class in $H^2(M; \mathbb{Z})$ with $x^*(a)=1$. Then there is a taut foliation $\mathcal{F}$ on $M$ whose Euler class is equal to $a$.
\end{eulerclassone}	

In \cite{yazdi2020thurston}, the second author constructed counterexamples to the above conjecture, assuming the \emph{fully marked surface theorem}. In \cite{gabai2020fullymarked}, David Gabai and the second author proved the fully marked surface theorem, and completed the negative answer to the Euler class one conjecture. Paraphrasing the negative answer to the above conjecture, there are infinitely many closed hyperbolic 3-manifolds $M$ for which some second cohomology class $a \in 2H^2(M;\mathbb{Z})$ with dual norm one is not realised as the Euler class of any taut foliation on $M$. 

Tight contact structures resemble taut foliations in many ways, while overtwisted contact structures are in analogy with foliations that have \emph{Reeb components}.  In particular, Eliashberg proved that a contact structure is tight if and only if it satisfies a relative version of Inequality \eqref{inequality}. Every $C^0$ taut foliation can be $C^0$-approximated by a tight contact structure, by the work of Eliashberg and Thurston \cite{eliashberg1998confoliations}, Kazez and Roberts \cite{kazez2015approximating}, and Bowden \cite{bowden2016approximating}. As a result, for any fixed closed, orientable, irreducible 3-manifold $M$, the set of Euler classes of tight contact structures on $M$ includes the set of Euler classes of taut foliations. There are examples of 3-manifolds for which the inclusion is strict; for example $S^3$ does not have any taut foliation \cite{novikov1965topology}, while it supports a unique tight contact structure \cite{bennequin1982entrelacements, eliashberg1992contact}. In \cite{yazdi2020thurston}, the second author asked whether the Euler class one conjecture has a positive answer for tight contact structures. 

Colin, Giroux, and Honda \cite{colin2009finitude} proved that on any closed oriented 3-manifold $M$, there are only finitely many homotopy classes of plane fields that carry a tight contact structure. Fixing a trivialization $\tau$ of $TM$, the homotopy class of a transversely oriented plane field $\xi$ over $M$ is determined by the data of a cohomology class $\Gamma_\tau(\xi) \in H^2(M ; \mathbb{Z})$ such that $2 \Gamma_\tau(\xi)$ is equal to the Euler class $e(\xi)$, and an element of the affine space $\mathbb{Z}/d(\xi)\mathbb{Z}$, where $d(\xi)$ is the divisibility of  $e(\xi)$ modulo torsion. See \cite[Proposition~4.1]{gompf1998handlebody}. Therefore, understanding possible values of Euler classes of tight contact structures would be a key step in classifying homotopy classes of tight contact structures on 3-manifolds. %\todo{MY: added this paragraph}

In this work, we investigate the constructed counterexamples in \cite{yazdi2020thurston}, and provide support for the Euler class one conjecture for tight contact structures. First, we need to briefly recall the setup from \cite[Section 4]{yazdi2020thurston}.

Let $S$ be a closed orientable surface of genus $g \geq 3$, and let $\gamma \subset S$ be a non-separating simple closed curve. Starting with a fibered hyperbolic 3-manifold $M_f$ with fiber $S$ and monodromy $f \colon S \rightarrow S$, the final manifold $M$ is the result of Dehn surgery on $\gamma \subset S \subset M_f$. The monodromy map $f$ and the Dehn surgery are chosen in a specific way to ensure that 
\begin{enumerate}
	\item $H_2(M_f)$ and $H_2(M)$ have respectively rank 1 and 2.
	\item $H_2(M;\mathbb{Z})$ is generated by surfaces $S$ and $F$, where $S$ is a copy of the fiber in $M$, and $F$ is a closed orientable surface of genus 2. Moreover, the unit ball of the Thurston norm and the dual Thurston norm are as in Figure \ref{unitballs}.
\end{enumerate}

\begin{figure}
\begin{tikzpicture}
\draw (-3,0) -- (3,0) node[below] {$(\frac{1}{2},0)$} (0,-0.85) -- (0,0.85) node[above] {$(0,\frac{1}{2g-2})$};
\draw[blue,densely dotted,very thick] (-2.5,0) coordinate (A) -- (0,0.5) coordinate (B) -- (2.5,0) coordinate (C) -- (0,-0.5) coordinate (D) -- cycle;
\foreach \e in {A,B,C,D} { \draw[fill=black] (\e) circle (0.075); }
\draw (4,0) coordinate (hmid) -- (6.75,0) node[below] {$(2,0)$} (5.375,-1.65) coordinate (vmid) -- (5.375,1.65) node[above] {$(0,2g-2)$};
\draw[blue,very thick] (4.5,-1.5) coordinate (G) rectangle (6.25,1.5) coordinate (H);
\foreach \e in {G,H,G|-H,H|-G,G|-hmid,H|-hmid,vmid|-G,vmid|-H} { \draw[fill=black] (\e) circle (0.075); }
\end{tikzpicture}
\caption{Unit ball of the Thurston norm (left), and the dual Thurston norm (right). The vertical and horizontal axes correspond to the surfaces $S$ and $F$, respectively.}
\label{unitballs}
\end{figure}
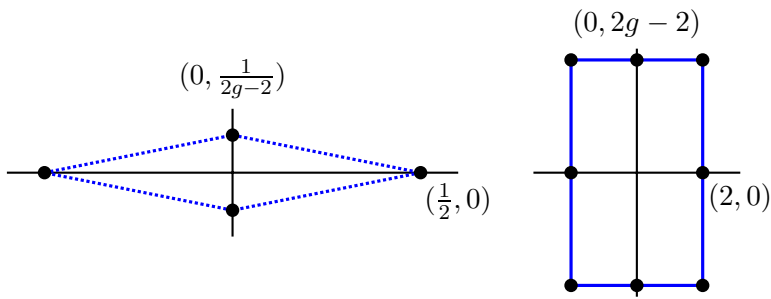

In particular, the point $(0,2-2g) \in 2H^2(M ; \mathbb{Z})$ has dual norm equal to 1. It is shown in \cite{yazdi2020thurston} that the class $(0,2-2g)$ is not realised as the Euler class of any taut foliation on $M$, hence providing counterexamples to the Euler class one conjecture. In this work, we show that the cohomology class $(0,2-2g)$ is indeed realised as the Euler class of a \emph{negative} tight contact structure on $M$, or equivalently of a (positive) tight contact structure on $-M$.  

\begin{thm}
	There is a weakly symplectically fillable contact structure $\xi_0$ on $-M$ that satisfies
	\[ \langle e(\xi_0), [S]\rangle = 2-2g \quad\text{and}\quad \langle e(\xi_0),[F]\rangle = 0. \]
	\label{thm: main}
\end{thm}

See Definition~\ref{def: weakly fillable} for the definition of a weakly symplectically fillable contact structure. In particular, the contact structure $\xi_0$ is tight. This shows that the Euler class one conjecture has a chance of being true for tight contact structures: in other words, every class of dual norm one could be the Euler class of a tight contact structure $\xi$ on $M$, where we allow $\xi$ to be either positive or negative.

\begin{remark}
The contact structure $\xi_0$ constructed in Theorem~\ref{thm: main} is virtually overtwisted, as explained in Remark~\ref{rem:virtually-ot}; this means that while $\xi_0$ is tight, it lifts to an overtwisted contact structure on some finite cover of $-M$.  By contrast, the contact structures that are $C^0$-close to taut foliations remain tight even when lifted to the universal cover \cite[Corollary~3.2.8]{eliashberg1998confoliations}.  It is therefore also interesting to ask whether the Euler class one conjecture holds for \emph{universally tight} contact structures; this stronger version of the conjecture might well be false.
\end{remark}

\begin{remark}
Another natural question is whether other points on the boundary of the unit ball are realised as Euler classes of tight contact structures. Gabai proved that for any compact orientable irreducible 3-manifold with boundary a (possibly empty) union of tori, the vertices of the dual unit ball are realised as Euler classes of taut foliations \cite{gabai1997problems}. For a proof see \cite{gabai2020fullymarked}. By the perturbation results of Eliashberg--Thurston, Kazez--Roberts, and Bowden, the vertices are realised as Euler classes of tight contact structures as well. Note that in \cite{yazdi2020thurston}, it was not determined whether the other points $(\pm2,2k)$ of the unit dual ball with $|k|<g-1$ are realised as Euler classes of taut foliations on $M$ or not. Likewise, here we do not answer the question of whether the points $(\pm2,2k)$ with $|k|<g-1$ are realised as Euler classes of tight contact structures on $M$ or not.
\end{remark} 

\subsection*{Acknowledgements} We thank the referee for thoughtful feedback on the initial version of this paper, and in particular for raising the question of whether the Euler class one conjecture holds for universally tight contact structures.  The second author acknowledges the support by an UKRI Postdoctoral Research Fellowship.

\section{Contact geometric background}

In this section we summarize the basic material we will need from contact geometry.  Much of this material is standard, but we include the details for the sake of foliations experts who may not be as familiar with contact geometry.
Useful references include the book by Geiges \cite{geiges-book}, various lecture notes by Etnyre \cite{etnyre-intro,etnyre-knots} and Massot \cite{massot-notes}, and the book by Ozbagci--Stipsicz \cite{ozbagci-stipsicz} for material on Stein manifolds.

\subsection{Basic definitions}

A 1-form $\alpha$ on a 3-manifold $Y$ is a \emph{contact form} if $\alpha \wedge d\alpha$ is nowhere zero, i.e., $\alpha$ is maximally non-integrable. A plane field $\xi \subset TY$ is a \emph{contact structure} if locally it can be defined by a contact 1-form $\alpha$ as $\xi = \ker \alpha$. It is \emph{coorientable} if $\xi$ is the kernel of a globally defined 1-form $\alpha$, or equivalently if the normal line bundle $\xi^\perp \subset TY$ is trivial.  When $Y$ is oriented and $\xi=\ker \alpha$ is coorientable, we say that $\xi$ is \emph{positive} if the orientation of $\alpha \wedge d\alpha$ agrees with that of $Y$. From now on, all contact structures are assumed to be positive and coorientable.

Contact structures satisfy a dichotomy which was introduced by Eliashberg, following Bennequin \cite{bennequin1982entrelacements}, for which now we introduce some terminology.
\begin{definition}
A knot $\lambda$ in a contact 3-manifold $(Y,\xi)$ is \emph{Legendrian} if it is everywhere tangent to $\xi$, meaning that $T\lambda \subset \xi|_\lambda$.
\end{definition}

Given a knot $L \subset Y$, by a \emph{framing} of $L$ we mean a trivialisation of the normal bundle of $L$ up to homotopy. This is equivalent to a choice of a non-zero section of the normal bundle up to homotopy, since the normal bundle is an oriented plane bundle. A Legendrian knot determines a \emph{contact framing} given by the normal to $L$ in $\xi$.  If a Legendrian knot $L$ bounds a compact orientable surface $S$, then the \emph{surface framing} on $L$ is defined by the normal to $L$ in $TS$.

\begin{definition}
An embedded disc $D \subset (Y, \xi)$ is an \emph{overtwisted disc} if $\partial D$ is a Legendrian knot and furthermore, the surface framing and the contact framing on $\partial D$ agree.

A contact structure $\xi$ is \emph{overtwisted} if there is an overtwisted disc $D \subset (Y,\xi)$, and otherwise $(Y,\xi)$ is called \emph{tight}.    
\end{definition}
Eliashberg \cite{eliashberg1989classification} proved that overtwisted contact structures are not so interesting: if $Y$ is a closed, oriented 3-manifold and $\Delta \subset Y$ an embedded disk, and if we fix $\xi_\mathrm{ot}$ on $\Delta$ to be a standard overtwisted disk, then the inclusion
\[ \left\{\begin{array}{c} \text{contact structures} \\ \xi\text{ on }Y \end{array} \,\middle\vert\, \xi|_\Delta = \xi_{\mathrm{ot}}\right\} \hookrightarrow \left\{ 2\text{-plane fields on }Y \right\} \]
is a homotopy equivalence.  We thus focus our attention on tight contact structures.

\subsection{Fillings and cobordisms}

We can sometimes identify tight contact structures with the following useful criterion.

\begin{definition}
A contact structure $(Y,\xi)$ is \emph{weakly symplectically fillable} if there is a compact symplectic 4-manifold $(W,\omega)$ with boundary $M$ such that $\omega|_\xi > 0$.
\label{def: weakly fillable}
\end{definition}

\begin{thm}[Gromov \cite{gromov-holomorphic}, Eliashberg \cite{eliashberg-filling}]
Weakly symplectically fillable contact manifolds are tight.
\end{thm}

We can also construct interesting symplectic cobordisms between contact manifolds; the following will suffice for our purposes.

\begin{definition}
A compact, almost complex 4-manifold $(W,J)$ is a \emph{Stein cobordism} if there is a function
\[ \phi: W \to [a,b] \]
such that
\begin{itemize}
\item $Y_0 = \phi^{-1}(a)$ and $Y_1 = \phi^{-1}(b)$ are regular level sets of $\phi$, with $\partial W = -Y_0 \sqcup Y_1$;
\item if $\omega = d(-d\phi\circ J)$, then $g_\phi(\cdot,\cdot) = \omega(\cdot, J\cdot)$ defines a Riemannian metric on $W$.
\end{itemize}
The latter condition says precisely that $\phi$ is \emph{strictly plurisubharmonic}.

In this case the 1-form $\alpha = -d\phi \circ J$ restricts to a contact form on each $Y_i$ ($i=0,1$), with contact structure $\xi_i = \ker(\alpha|_{Y_i}) = TY_i \cap JTY_i$, and $\omega=d\alpha$ is a symplectic form on $W$.
\end{definition}

\begin{remark}
A \emph{Stein filling} is a Stein cobordism $(W,J)$ from the empty manifold to $(Y,\xi)$.  Stein fillings are special cases of weak symplectic fillings: given any nonzero $v\in\xi_p$, the pair $(v,Jv)$ is an oriented basis of $\xi_p = T_pY \cap J T_pY$, and we have $\omega(v,Jv) = g_\phi(v,v) > 0$, so $\omega|_\xi > 0$.  
\end{remark}

\begin{ex} \label{ex:symplectization}
Let $(Y,\xi)$ be a contact 3-manifold with contact form $\alpha$ and associated \emph{Reeb vector field} $R_\alpha$, meaning that $\iota_{R_\alpha} d\alpha=0$ and $\alpha(R_\alpha)=1$.  Let $W = [0,1]\times Y$ with interval coordinate $t$.  Then we can give $W$ a $t$-invariant almost complex structure $J$, with $J\xi=\xi$ and $J(\partial_t)=R_\alpha$ (note that then $J(R_\alpha) = J^2(\partial_t) = -\partial_t$), and the function
\[ \phi(x,t) = e^t \]
is strictly plurisubharmonic; it satisfies $-d\phi\circ J = e^t\alpha$, which restricts to a contact form for $\xi$ on both $\{0\}\times Y$ and $\{1\}\times Y$.  Thus $J$ is a Stein structure on the symplectization $(W,d(e^t\alpha))$, realizing it as a Stein cobordism from $(Y,\xi)$ to itself.
\end{ex}

\begin{prop} \label{prop:stein-euler-class}
If $(W,J)$ is a Stein cobordism from $(Y_0,\xi_0)$ to $(Y_1,\xi_1)$, then each $\xi_i$ has Euler class
\[ e(\xi_i) = c_1(TW,J) |_{Y_i} \text{\ \ \ for }i=0,1. \]
\end{prop}

\begin{proof}
Let $\phi: W \to \R$ be a strictly plurisubharmonic function, so that $Y_i$ has contact form $\alpha = -d\phi\circ J$.  Then on $Y_i$ we note that $\nabla\phi$ is nonzero but not tangent to $Y_i$, since $Y_i$ is a regular level set of $\phi$, and likewise that $J\nabla\phi \not\in \xi_i$ since
\[ \alpha(J\nabla\phi) = (-d\phi \circ J)(J\nabla\phi) = d\phi(\nabla\phi) = |\nabla\phi|^2 > 0. \]
Thus we have an isomorphism of complex vector bundles
\[ TW|_{Y_i} \cong \xi_i \oplus \left(\C\cdot\nabla\phi\right) \cong \xi_i \oplus \underline{\C}, \]
and it follows as claimed that
\[ c_1(TW,J)|_{Y_i} = c_1(\xi_i,J) \oplus \underbrace{c_1(\underline{\C})}_{=0} = c_1(\xi_i,J) = e(\xi_i). \qedhere \]
\end{proof}

\subsection{Legendrian surgery and the Euler class}

Let $\lambda \subset (Y,\xi)$ be a Legendrian knot.  Then we can construct a new contact manifold $(Y',\xi')$ by \emph{Legendrian surgery} \cite{eliashberg-stein,weinstein-surgery}.  Smoothly, this amounts to a Dehn surgery on $\lambda$ of slope one less than its contact framing, meaning that we obtain it from the contact framing by adding one left-handed twist.  It moreover comes with a natural contact structure, which is tight on the solid torus that realizes the surgery and which agrees with the one on $Y$ outside a neighbourhood of $\lambda$.

\begin{lem} \label{lem:stein-cobordism}
Let $\lambda \subset (Y,\xi)$ be a Legendrian knot, and form $(Y',\xi')$ by Legendrian surgery on $\lambda$.  Then there is a Stein cobordism $(W,J)$ from $(Y,\xi)$ to $(Y',\xi')$, where $W$ is obtained smoothly by attaching a 2-handle to $[0,1]\times Y$
along $\{1\}\times\lambda$.  Moreover, if $(Y,\xi)$ is weakly symplectically fillable, then so is $(Y',\xi')$.
\end{lem}

\begin{proof}
We put a Stein structure on $[0,1]\times Y$ as in Example~\ref{ex:symplectization}.  Then Eliashberg \cite{eliashberg-stein} shows how to extend the Stein structure across a 4-dimensional 2-handle which has been attached to $\{1\}\times\lambda$ with framing one less than the contact framing.  This gives the desired Stein cobordism.

If $(Y,\xi)$ is weakly fillable then Etnyre and Honda \cite[Theorem~2.5]{eh-no-fillings} proved that $(Y',\xi')$ will be as well.  The key observation is that extending a symplectic form across the 2-handle only uses the symplectic structure near $\lambda$, so they deform it slightly to resemble a strong symplectic filling near $\lambda$ and then proceed with Eliashberg's construction.
\end{proof}

We would like to use the Stein cobordism $(W,J)$ of Lemma~\ref{lem:stein-cobordism} to compute the Euler class $e(\xi')$, but we must first figure out how to compute $c_1(TW,J)$.  For Stein manifolds $(W,J)$ built by attaching 1- and 2-handles to the standard 4-ball, Gompf \cite[Proposition~2.3]{gompf1998handlebody} explicitly determined the evaluation of $c_1(TW,J)$ on a basis of $H_2(W)$.

\begin{definition}
Let $\lambda \subset (Y,\xi)$ be an oriented, nullhomologous Legendrian knot, with Seifert surface $\Sigma \subset Y$.  We can then fix a trivialization
\[ \tau: \xi|_\Sigma \xrightarrow{\cong} \Sigma \times \R^2. \]
The oriented unit tangent vectors to $\lambda$ define a section of $\xi|_\lambda$, which pulls back to a nonzero section
\[ s: \partial \Sigma \to \xi|_{\partial\Sigma} \xrightarrow{\tau} \partial \Sigma \times \R^2, \]
and the \emph{rotation number} $\rot(\lambda,[\Sigma])$ is defined as the winding number of $s$ in $\R^2\setminus \{0\}$.  It depends only on the relative homology class of $\Sigma$.
\end{definition}

An oriented Legendrian knot $\lambda$ admits both a positive stabilization $S_+(\lambda)$ and a negative stabilization $S_-(\lambda)$, as shown in Figure~\ref{fig:stabilization}.
\begin{figure}
\begin{tikzpicture}
\draw[->] (-1.5,2) -- node[midway,above] {$\lambda$} ++(3,0);
\draw[->] (-4,0) to[out=0,in=180] ++(2,0.25) -- ++(-1,-0.5) to[out=0,in=180] ++(2,0.25);
\draw[->] (1,0) to[out=0,in=180] ++(2,-0.25) -- ++(-1,0.5) to[out=0,in=180] ++(2,-0.25);
\node[below] at (-2.5,-0.25) {$S_+(\lambda)$};
\node[below] at (2.5,-0.25) {$S_-(\lambda)$};
\draw[densely dashed,thin,->] (-0.75,1.75) -- (-2.25,0.5);
\draw[densely dashed,thin,->] (0.75,1.75) -- (2.25,0.5);
\end{tikzpicture}
\caption{Stabilization of an oriented Legendrian knot $\lambda$.  Here we work locally, in the standard contact $(\R^3,\xi=\ker(dz-ydx))$, with each knot shown in the front ($xz$-) projection.}
\label{fig:stabilization}
\end{figure}
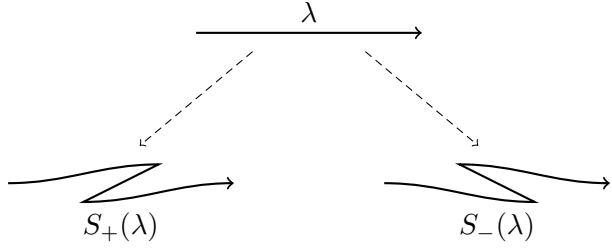
This operation adds a left-handed twist to the contact framing while changing the rotation numbers by
\[ \rot(S_\pm(\lambda),[\Sigma]) = \rot(\lambda,[\Sigma]) \pm 1. \]

The following result of Gompf, which is essentially \cite[Proposition~2.3]{gompf1998handlebody}, says that rotation numbers determine the first Chern classes of the sorts of Stein cobordisms described by Lemma~\ref{lem:stein-cobordism}.

\begin{lem} \label{lem:compute-c1}
Let $\lambda \subset (Y,\xi)$ be a nullhomologous Legendrian knot, with Seifert surface $\Sigma$.  Let $(Y',\xi')$ be the result of Legendrian surgery on $\lambda$, with
\[ (W,J): (Y,\xi) \to (Y',\xi') \]
the Stein 2-handle cobordism described by Lemma~\ref{lem:stein-cobordism}.  Form a closed smooth surface $\hat\Sigma \subset W$ by gluing a core of the 2-handle to $\Sigma$, where we have embedded the latter in $Y\times \{1\}$ and then pushed its interior into $Y\times (0,1)$.  Then
\[ \langle c_1(TW,J), [\hat\Sigma] \rangle = \rot(\lambda,[\Sigma]). \]
\end{lem}

The proof of \cite[Proposition~2.3]{gompf1998handlebody} is carried out specifically for Stein manifolds built by attaching 1- and 2-handles to the standard 4-ball, but it applies equally well in this setting because it only requires a trivialization of $\xi$ on $\Sigma$.

\section{The candidate manifolds}

\begin{notation}
	For a topological space $X$ and $A \subset X$, denote the interior of $A$ by $A^\circ$.
\end{notation}

We recall the construction of the manifolds from \cite{yazdi2020thurston}. Let $S$ be a closed orientable surface of genus $g$, where either $g=3$ or $g \geq 6$, and $\gamma$ be an oriented non-separating simple closed curve in $S$. Let $A$ be a tubular neighbourhood of $\gamma$ in $S$ with $\partial A = \{ \gamma_+, \gamma_- \}$. The orientation on $\gamma$ induces an orientation on $\gamma_+$ and $\gamma_-$. Pick points $p_1$, $p_2$, $p_3$ on $\gamma_+$ such that the cyclic ordering of $p_i$ agrees with the orientation of $\gamma_+$. Similarly pick points $q_1$, $q_2$, $q_3$ on $\gamma_-$. Pick disjoint oriented simple arcs $\ell_i \subset A$ such that $\ell_i$ connects $q_i$ to $p_i$, where the index $i$ varies in $\mathbb{Z}/3\mathbb{Z}$ throughout. Similarly pick disjoint oriented simple arcs $m_i \subset A$ connecting $q_{i+1}$ to $p_i$ such that $m_i$ is obtained by perturbing the union of the two oriented arcs $q_{i+1}q_i$ and $l_i$. See Figure \ref{cylinder}.

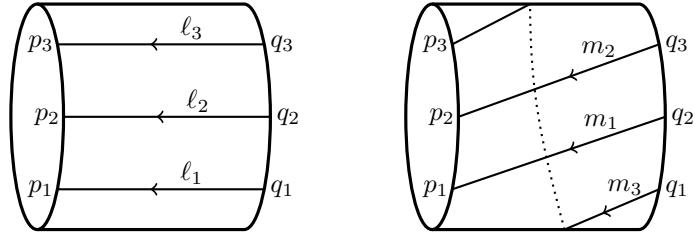
\begin{figure}
\begin{tikzpicture}
\begin{scope}[every path/.style={very thick}]
\foreach \x in {0,5.25} {
  \draw (\x,0) ellipse (0.35 and 1.5);
  \draw (\x,1.5) -- ++(2.75,0) arc (90:-90:0.35 and 1.5) -- ++(-2.75,0);
}
\end{scope}
\begin{scope}[decoration={markings,mark=at position 0.55 with {\arrow{>}}}] 
\foreach \i/\t in {3/40,2/0,1/-40} {
  \path (0,1.5) arc (90:\t:0.35 and 1.5) coordinate (p\i) -- node[above,pos=0.65,inner sep=1.5pt] {\small$\ell_{\i}$} ++(2.75,0) coordinate (q\i);
  \draw[postaction={decorate}] (q\i) node[right,inner sep=2pt] {\small$q_{\i}$} -- (p\i) node[left,inner sep=1pt] {\small$p_{\i}$};
}
\end{scope}
\begin{scope}
\foreach \i/\t in {3/40,2/0,1/-40} {
  \path (5.25,1.5) arc (90:\t:0.35 and 1.5) coordinate (p\i) -- ++(2.75,0) coordinate (q\i);
  \path (q\i) node[right,inner sep=2pt] {\small$q_{\i}$} -- (p\i) node[left,inner sep=1pt] {\small$p_{\i}$};
}
\coordinate (p0) at (5.25,-2.25);
\coordinate (q4) at (8,2.25);
\draw[decoration={markings,mark=at position 0.45 with {\arrow{>}}},postaction={decorate}] (q3) -- node[above,pos=0.3,inner sep=3pt] {\small$m_2$} (p2);
\draw[decoration={markings,mark=at position 0.45 with {\arrow{>}}},postaction={decorate}] (q2) -- node[above,pos=0.3,inner sep=3pt] {\small$m_1$} (p1);
\clip (5.25,-1.5) rectangle (8.35,1.5);
\draw[decoration={markings,mark=at position 0.25 with {\arrow{>}}},postaction={decorate}] (q1) -- node[above,pos=0.15,inner sep=3pt] {\small$m_3$} (p0);
\path[name path=bottomarc] (q1) -- (p0);
\path[name path=bottomedge] (5.25,-1.5) -- (8,-1.5);
\path [name intersections={of=bottomarc and bottomedge,by=r1}];
\draw[name path=toparc] (q4) -- (p3);
\path[name path=topedge] (5.25,1.5) -- (8,1.5);
\path[name intersections={of=toparc and topedge,by=r2}];
\draw[dotted] (r1) to[out=105,in=270] (r2);
\end{scope}
\end{tikzpicture}
\caption{Left: arcs $\ell_i \subset A$. Right: arcs $m_i \subset A$.}
\label{cylinder}
\end{figure}

Let $\delta_1$, $\delta_2$, $\delta_3$ be disjoint oriented simple arcs in $S-A^\circ$ with $\delta_i$ joining $p_i$ to $q_{i-1}$. Define the curve $\alpha$ as the union of the six oriented arcs $\ell_i$ and $\delta_i$. Define $\beta$ as the union of the six oriented arcs $m_i$ and $\delta_i$. Both $\alpha$ and $\beta$ are connected simple closed curves. 

Let $M_f$ be the mapping torus of $f$. Hence, $M_f = S \times [0,1] / \sim$ where $(x, 1) \sim (f(x),0)$ for every $x \in S$. It is shown in \cite[Lemma 4.1]{yazdi2020thurston} that for a suitable choice of arcs $\delta_i$, there is a pseudo-Anosov map $f \colon S \rightarrow S$ such that 
\begin{enumerate}
	\item $H_2(M_f)$ has rank 1. 
	\item $f(\alpha)=\beta$.
\end{enumerate}
Let $N$ be a solid torus with meridional disc $D$. Let $U = A \times [\frac{1}{4}, \frac{3}{4}]$ be a regular neighbourhood of $\gamma \times \{\frac{1}{2}\}$ in $M_f$. The manifold $M$ is defined by Dehn surgery on $U \subset M_f$, where the glued-in solid torus is identified with $N$, and $\partial D$ is attached to the union $C$ of $\ell_i \times \{ \frac{3}{4} \}$, $m_i \times \{ \frac{1}{4} \}$, $p_i \times [\frac{1}{4}, \frac{3}{4}]$, and $q_i \times [\frac{1}{4}, \frac{3}{4}]$ for $1 \leq i \leq 3$. See Figure~\ref{boundaryofD}.

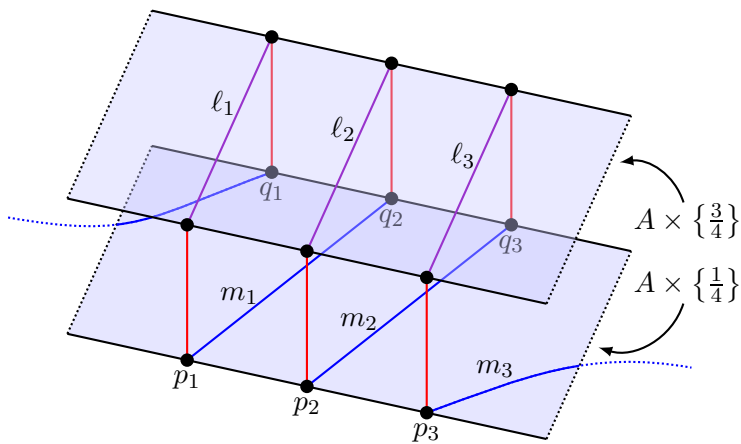
\begin{figure}
\definecolor{ellpurple}{rgb}{0.6,0.2,0.8}
\begin{tikzpicture}
% lower rectangle
\begin{scope}[xslant=0.45,yslant=-0.2]
\path[fill=blue!10] (0,0) rectangle (7,2.5);
\draw (0,0) coordinate (sw) -- ++(7,0) coordinate (se) (0,2.5) coordinate (nw) -- ++(7,0) coordinate (ne);
\draw[densely dotted] (0,0) -- ++(0,2.5) (7,0) -- node[midway,outer sep=0] (A1arrow) {} ++(0,2.5);
\foreach \i in {1,2,3} {
  \coordinate (p\i) at (1.75*\i,0);
  \coordinate (q\i) at (1.75*\i,2.5);
}
\coordinate (endp) at (8.5,1.25);
\coordinate (endq) at (-1.5,1.25);
\end{scope}
\draw[blue] (p1) -- node[pos=0.4,left] {\textcolor{black}{$m_1$}} (q2);
\draw[blue] (p2) -- node[pos=0.4,left] {\textcolor{black}{$m_2$}} (q3);
\draw[blue,densely dotted] (p3) to[out=20,in=170] (endp);
\draw[blue,densely dotted] (q1) to[out=200,in=-10,looseness=1.25] (endq);
\begin{scope}
\clip (sw) -- (se) -- (ne) -- (nw) -- cycle;
\draw[blue] (p3) to[out=20,in=170] node[pos=0.25,above] {\textcolor{black}{$m_3$}} (endp);
\draw[blue] (q1) to[out=200,in=-10,looseness=1.25] (endq);
\end{scope}
\foreach \i in {1,2,3} {
  \foreach \j in {p,q} {
    \draw[red] (\j\i) -- ++(0,1.8);
    \draw[fill=black] (\j\i) node[below] {$\j_{\i}$} circle (0.075);
  }
}
% upper rectangle
\begin{scope}[yshift=1.8cm]
\begin{scope}[xslant=0.45,yslant=-0.2]
\path[fill=blue!20,fill opacity=0.4] (0,0) rectangle (7,2.5);
\draw (0,0) -- ++(7,0) (0,2.5) -- ++(7,0);
\draw[densely dotted] (0,0) -- ++(0,2.5) (7,0) -- node[pos=0.75,outer sep=0] (A3arrow) {} ++(0,2.5);
\foreach \i in {1,2,3} {
  \coordinate (p\i) at (1.75*\i,0);
  \coordinate (q\i) at (1.75*\i,2.5);
}
\end{scope}
\end{scope}
\foreach \i in {1,2,3} {
  \draw[ellpurple] (p\i) -- node[pos=0.65,left,inner sep=2pt] {$\textcolor{black}{\ell_{\i}}$} (q\i);
  \foreach \j in {p,q} { \draw[fill=black] (\j\i) circle (0.075); }
}
% label rectangles
\node[inner sep=0] (A1head) at (8.25,0.65) {$A\times\left\{\tfrac{1}{4}\right\}$};
\draw[-latex] (A1head) to[bend left=45] (A1arrow);
\node[inner sep=0] (A3head) at (8.25,1.5) {$A\times\left\{\tfrac{3}{4}\right\}$};
\draw[-latex] (A3head) to[bend right=40] (A3arrow);
\end{tikzpicture}
\caption{The curve $\partial D$ is identified with the union of $\ell_i \times \{ \frac{3}{4} \}$ (purple), $m_i \times \{ \frac{1}{4} \}$ (blue), $p_i \times [\frac{1}{4}, \frac{3}{4}]$ (red), and $q_i \times [\frac{1}{4}, \frac{3}{4}]$ (red) on the boundary of $M_f - U^\circ$. For readability, we label $\ell_i \times \{ \frac{3}{4} \}$ and $m_i \times \{\frac{1}{4} \}$ with $\ell_i$ and $m_i$ respectively.}
\label{boundaryofD}
\end{figure}

Define the surface $F_0 \subset M$ as the union of $D$ and the three bands $\delta_i \times [\frac{1}{4}, \frac{3}{4}]$. Then $F_0$ is a surface of genus 1 with two boundary components $\alpha \times \{\frac{3}{4}\}$ and $- \beta \times \{\frac{1}{4}\}$. Define the surface $F \subset M$ as the union of $F_0$, the two annuli $\alpha \times [\frac{3}{4}, 1]$ and $\beta \times [0, \frac{1}{4}]$, and identifying $\alpha \times \{1\}$ with $\beta \times \{0\}$ by the map $f$. Then $F$ is a closed orientable surface of genus 2. It is shown in \cite[Lemma 5.1]{yazdi2020thurston} that $H_2(M; \mathbb{Z})$ is generated by $S$ and $F$, and the unit ball and the unit dual ball of the Thurston norm are as in Figure \ref{unitballs}.

\section{Construction of tight contact structures with the given Euler classes}

Our goal is to construct a tight contact structure $\xi$ on $-M$ with Euler class
\[ a=(0,2-2g) \in 2H^2(-M;\mathbb{Z}). \]
Concretely, this means that we expect it to satisfy the relations
\begin{align*}
\langle e(\xi), [S] \rangle &= 2-2g, &
\langle e(\xi), [F] \rangle &= 0.
\end{align*}
We will rely heavily on the theory of \emph{convex surfaces} in contact 3-manifolds, as developed by Giroux \cite{giroux-convexity}; see also the lecture notes by Massot \cite{massot-notes}.

To begin, we choose an essential simple closed curve $c \in S$ that is both disjoint from $\gamma$ and nonseparating in $S \setminus \gamma$, and we fix a collar neighbourhood
\[ V \cong c \times [-1,1] \]
of $c$ in $S$ which is also disjoint from $\gamma$.  The following is then a straightforward consequence of results of Honda, Kazez, and Mati\'{c} \cite{hkm-hyperbolic}.

\begin{lem} \label{lem:contact-torus}
Up to isotopy, there is a unique tight contact structure $\xi_{f^{-1}}$ on the mapping torus
\[ M_{f^{-1}} = \frac{S \times [0,1]}{(x,1) \sim (f^{-1}(x),0)} \]
of $f^{-1}$, such that the middle fiber
\[ S_{1/2} = S\times\{\tfrac{1}{2}\} \]
is a convex surface with dividing set $\Gamma_{S_{1/2}} = \partial V \times \{\frac{1}{2}\}$ and negative region $(S_{1/2})_- = V\times\{\frac{1}{2}\}$.  If $S$ is any fiber of $M_{f^{-1}}$ then
\[ \langle e(\xi_{f^{-1}}), [S] \rangle = 2 - 2g, \]
and $\xi_{f^{-1}}$ is universally tight and weakly symplectically fillable.
\end{lem}

\begin{proof}
We note that $f(c)$ is not isotopic or even homologous to $c$, since $b_1(M_f) = 1$.  With this in mind, Honda, Kazez, and Mati\'c \cite[Theorem~1.1]{hkm-hyperbolic} proved that there is a unique universally tight contact structure $\xi_I$ on $S \times [0,1]$ which satisfies the following:
\begin{itemize}
\item the surfaces $S_1 = S \times \{1\}$ and $S_0 = S \times \{0\}$ are convex;
\item the dividing set on $S_1$ is
\[ \Gamma_{S_1} = f(\partial V) \times \{1\} = f(c\times\{\pm1\}) \times \{1\}, \]
oriented so that the negative region $(S_1)_-$ is the annulus $f(V)\times\{1\}$;
\item the dividing set on $S_0$ is
\[ \Gamma_{S_0} = \partial V \times \{0\} = (c\times\{\pm1\}) \times \{0\}, \]
oriented so that the negative region $(S_0)_-$ is the annulus $V\times\{0\}$;
\item the relative Euler class \cite[\S3]{hkm-hyperbolic} of $\xi_I$ satisfies $\mathrm{PD}(\tilde{e}(\xi_I)) = c - f(c)$.
\end{itemize}
Since $S_0$ is convex, it has a collar neighborhood $S \times [0,\epsilon)$ on which the contact structure $\xi_I$ is invariant in the $[0,\epsilon)$ direction; we will reparametrise the interval factor $[0,1]$ smoothly so that we can take $\epsilon = \frac{2}{3}$, and then each surface $S_t = S \times \{t\}$ with $0 \leq t < \frac{2}{3}$ is convex with negative region $(S_t)_- = V \times \{t\}$.

We can now glue $S_1$ to $S_0$ via $f^{-1}$, which identifies $\beta \times \{1\} \subset S_1$ with $\alpha \times \{0\} \subset S_0$, and $\Gamma_{S_1} \subset S_1$ with $\Gamma_{S_0} \subset S_0$.  Then $\xi_I$ descends to a contact structure $\xi_{f^{-1}}$ on $M_{f^{-1}}$, which is universally tight, as argued in the proof of \cite[Theorem~1.2]{hkm-hyperbolic}.  The evaluation of its Euler class $e(\xi_{f^{-1}})$ on the class of a fiber does not depend on the choice of fiber, so by convex surface theory we compute that
\[ \langle e(\xi_{f^{-1}}), [S_{1/2}] \rangle = \underbrace{\chi((S_{1/2})_+)}_{=2-2g} - \underbrace{\chi((S_{1/2})_-)}_{=0} = 2-2g. \]
Thus $\xi_{f^{-1}}$ is ``extremal'' in the sense of \cite{hkm-hyperbolic}, and then \cite[Theorem~1.2]{hkm-hyperbolic} asserts that it is unique up to isotopy and also weakly fillable.
\end{proof}

\begin{lem} \label{lem:surgery-slope}
There is an orientation-preserving diffeomorphism between $-M$ and the manifold built by $-3$-framed Dehn surgery on the curve
\[ \gamma \times \{\tfrac{1}{2}\} \subset M_{f^{-1}}, \]
with respect to the framing given by a push-off inside the surface $S_{1/2} = S\times\{\frac{1}{2}\}$.
\end{lem}

\begin{proof}
We briefly recall the construction of $M$ from the mapping torus $M_f$.  In $M_f$, we identified a regular neighbourhood $U$ of $\gamma\times\{\frac{1}{2}\}$, and there is a natural basis $(\mu_f,\lambda_{f,S})$ of $H_1(\partial U)\cong \mathbb{Z}^2$ in which $\mu_f$ is a meridian (i.e., it is the oriented boundary of a meridional disk inside $U$) and $\lambda_{f,S}$ is a push-off of $\gamma\times\{\frac{1}{2}\}$ inside $S \times \{\frac{1}{2}\}$.  In these coordinates, we formed $M$ from a $+3$-surgery on $\gamma\times\{\frac{1}{2}\}$, meaning that we performed Dehn filling along a curve $\alpha$ in the homology class $3[\mu_f]+[\lambda_{f,S}] \in H_1(\partial (M_f\setminus U^\circ))$.

Similarly, we can produce $-M$ by Dehn surgery on the curve $\gamma\times\{\frac{1}{2}\}$ in $M_{f^{-1}}$.  We start with an orientation-preserving diffeomorphism
\[ \varphi: -M_f \xrightarrow{\sim} M_{f^{-1}}, \]
built by sending $(x,t) \in S \times [0,1]$ to $(x,1-t) \in S \times [0,1]$ and then gluing $S\times \{1\}$ to $S\times \{0\}$ on either side to form the mapping tori.  Then $-M$ is the result of Dehn filling $M_{f^{-1}} \setminus \varphi(U^\circ)$ along the curve $\varphi(\alpha)$.  The image $\varphi(U^\circ)$ is an open tubular neighborhood of
\[ \varphi(\gamma\times\{\tfrac{1}{2}\}) = \gamma\times\{\tfrac{1}{2}\}, \]
with oriented meridian $\mu_{f^{-1}}$ and push-off $\lambda_{f^{-1},S}$ of $\gamma\times\{\frac{1}{2}\}$ inside $S_{1/2} \subset M_{f^{-1}}$.  Then we have
\[ \varphi_*([\mu_f]) = -[\mu_{f^{-1}}] \quad\text{and}\quad \varphi_*([\lambda_{f,S}]) = [\lambda_{f^{-1},S}] \]
as elements of $H_1(\varphi(\partial U))$, and so
\[ [\varphi(\alpha)] = -3[\mu_{f^{-1}}] + [\lambda_{f^{-1},S}] \]
in $H_1(\varphi(\partial U))$.  This says that $-M$ is the result of $-3$-surgery on $\gamma\times\{\frac{1}{2}\} \subset M_{f^{-1}}$, with respect to the framing given by $S_{1/2}$, and $\varphi$ extends across the glued-in solid tori to give the desired diffeomorphism.
\end{proof}

\theoremstyle{theorem}
\newtheorem*{main}{Theorem \ref{thm: main}}
\begin{main}
	There is a weakly symplectically fillable contact structure $\xi_0$ on $-M$ that satisfies
	\[ \langle e(\xi_0), [S]\rangle = 2-2g \quad\text{and}\quad \langle e(\xi_0),[F]\rangle = 0. \]
\end{main}

\begin{proof}
Let $\xi_{f^{-1}}$ be the contact structure on the mapping torus $M_{f^{-1}}$ produced by Lemma~\ref{lem:contact-torus}; we continue to write

\[ \Gamma_{S_{1/2}} = \partial V \times \{\tfrac{1}{2}\} \]
for the dividing set of the convex surface $S_{1/2}$.  Since the curve 
\[ \gamma \times \{\tfrac{1}{2}\} \subset S_{1/2} \subset M_{f^{-1}} \]
does not separate the complement $S_{1/2} \setminus \Gamma_{S_{1/2}}$, it is \emph{nonisolating}, and so we can apply the Legendrian realisation principle \cite[Theorem~3.7]{honda-classification1} to find a $C^\infty$-small isotopy of $S_{1/2}$ through convex surfaces in $(M_{f^{-1}},\xi_{f^{-1}})$ which carries $\gamma \times \{\frac{1}{2}\}$ to a Legendrian curve.

Having done this, the new dividing curves on $S_{1/2}$ are a small perturbation of the original dividing curves $c\times\{\pm1\}$, so the newly Legendrian $\lambda = \gamma\times\{\frac{1}{2}\}$ is still disjoint from them.  We know that $\lambda = \gamma \times \{\frac{1}{2}\}$ is nullhomologous in $M_{f^{-1}}$: we can orient it so that $T = F \setminus D^\circ$ is a Seifert surface in $M_{f^{-1}} \setminus U^\circ \cong -M \setminus N^\circ$.  Then $\lambda$ has a rotation number
\[ r_0 = \rot(\lambda,[T]). \]
The contact framing of $\lambda$, defined as the orthogonal complement of $T\lambda$ inside $\xi|_{\lambda}$, is moreover given by a geometric count of intersection points:
\[ \mathrm{tw}\left(\lambda, S_{1/2}\right) = -\frac{\#( \Gamma_{S_{1/2}} \cap \lambda)}{2} = 0. \]
We stabilize $\lambda$ twice in several different ways to produce three new Legendrians, each of them smoothly isotopic to $\lambda$:
\begin{align*}
\lambda_- &= S_-(S_-(\lambda)), &
\lambda_0 &= S_+(S_-(\lambda)), &
\lambda_+ &= S_+(S_+(\lambda)).
\end{align*}
Their rotation numbers satisfy
\begin{align*}
\rot(\lambda_-,[T]) &= r_0-2, \\
\rot(\lambda_0,[T]) &= r_0, \\
\rot(\lambda_+,[T]) &= r_0+2,
\end{align*}
and in particular they are all different.

Since each stabilization lowers the contact framing by one, it follows that a $-1$-surgery on any of $\lambda_-,\lambda_0,\lambda_+$ with respect to the contact framing is smoothly equivalent to a $-3$-surgery on $\lambda$ with respect to the $S_{1/2}$-framing, and thus by Lemma~\ref{lem:surgery-slope} it produces $-M$.  We let
\[ \xi_-, \ \xi_0,\ \xi_+ \]
denote the contact structures on $-M$ resulting from Legendrian surgery on $\lambda_-$, $\lambda_0$, and $\lambda_+$ in $(M_{f^{-1}},\xi_{f^{-1}})$, respectively.  By Lemma~\ref{lem:stein-cobordism} we know that each of these is weakly symplectically fillable, and moreover that if $W$ is the trace of this surgery, viewed as a smooth cobordism from $M_{f^{-1}}$ to $-M$ built by attaching a 2-handle to $M_{f^{-1}} \times [0,1]$, then it admits three different Stein structures
\[ J_-,\ J_0,\ J_+ \]
which realize it as a Stein cobordism from $(M_{f^{-1}},\xi_{f^{-1}})$ to $-M$ equipped with $\xi_-$, $\xi_0$, and $\xi_+$ respectively.  These moreover satisfy $e(\xi_-) = c_1(TW,J_-)|_{-M}$ and likewise for $\xi_0$ and $\xi_+$.

It remains to compute the Euler class of each contact structure $\xi_\circ$, where $\circ \in \{-,0,+\}$.  For the image of a fiber surface $S$, we can take the convex surface $S_0 = S\times\{0\}$ used in Lemma~\ref{lem:contact-torus}, which remains undisturbed by the surgery operation above, to compute that
\[ \langle e(\xi_\circ), [S] \rangle = \underbrace{\chi((S_0)_+)}_{=2-2g} - \underbrace{\chi((S_0)_-)}_{=0} = 2-2g. \]
The evaluation on $[F]$ requires only slightly more effort: Proposition~\ref{prop:stein-euler-class} and Lemma~\ref{lem:compute-c1} say that
\[ \langle e(\xi_\circ), [F] \rangle = \langle c_1(TW,J_\circ), [F] \rangle = \rot(\lambda_\circ,[T]). \]
Now since the surface $F \subset -M$ has genus 2 and each contact structure $\xi_\circ$ is weakly fillable (in particular, tight), we have for each $\circ\in\{-,0,+\}$ the inequality
\[ \left| \langle e(\xi_\circ), [F] \rangle \right| \leq -\chi(F) = 2, \]
due to Eliashberg \cite[Theorem~2.2.1]{eliashberg1992contact}.  But then
\[ \rot(\lambda_+,[T]) = r_0+2 \quad\text{and}\quad \rot(\lambda_-,[T]) = r_0-2 \]
must be between $-2$ and $2$ inclusive, and this is only possible if $r_0 = 0$.  We conclude that
\[ \langle e(\xi_0),[F] \rangle = \rot(\lambda_0,[T]) = r_0 = 0 \]
and so $\xi_0$ is the desired contact structure.
\end{proof}

\begin{remark} \label{rem:virtually-ot}
The contact structure $\xi_0$ of Theorem~\ref{thm: main} is virtually overtwisted.  To see this, we note that it is constructed by Legendrian surgery on a knot $\lambda_0 = S_+(S_-(\lambda))$, which has one stabilization of each sign.  If we let $\mu$ be the image in $-M$ of a meridian of $\lambda_0$, then the lift of $\xi_0$ to any cover $\tilde{M} \to -M$ is overtwisted as long as $\mu$ is not in the image of the inclusion $\pi_1(\tilde{M}) \to \pi_1(-M)$, by the proof of \cite[Proposition~5.1]{gompf1998handlebody}.  Such covers always exist since $\mu$ has intersection number $\pm1$ with the surface $F$ and is therefore homologically essential.
\end{remark}

\bibliographystyle{alpha}
\bibliography{Reference-foliation-and-contact}

\end{document}